\documentclass[a4paper]{article}
\usepackage[utf8]{inputenc}

\usepackage[top=2.5cm, bottom=2.5cm, left=2.5cm, right=2.5cm]{geometry}
\usepackage{amsmath, amssymb}
\usepackage{mathrsfs}
\usepackage[hidelinks]{hyperref}
\usepackage[shortlabels]{enumitem}

 \usepackage{amsthm}
 \newtheorem{thm}{Theorem}[section]
 \newtheorem{lm}[thm]{Lemma}
 
 \newtheorem{crl}[thm]{Corollary}
 \newtheorem{prop}[thm]{Proposition}

 \theoremstyle{definition}
 
 \newtheorem{rmk}[thm]{Remark}
 \newtheorem{df}[thm]{Definition}
 \newtheorem{constr}[thm]{Construction}

 \newcommand{\eps}{\varepsilon}

 \newcommand{\RR}{\mathbb R}
 \newcommand{\FF}{\mathbb F}
 \newcommand{\CC}{\mathbb C}

 \newcommand{\vspan}[1]{\left \langle #1 \right \rangle}
 \newcommand{\set}[1]{ \left \{ #1 \right \} }
 \newcommand{\sett}[2]{ \left\{ #1 \, \, || \, \, #2 \right \} }

 \newcommand{\pg}{\textnormal{PG}}
 
 \newcommand{\one}{\mathbf 1}
 \newcommand{\zero}{\mathbf 0}
 
 \newcommand{\pgl}{\textnormal{PGL}}
 \newcommand{\psl}{\textnormal{PSL}}

 \renewcommand{\mp}{\mathcal P}
 \newcommand{\mb}{\mathcal B}
 \newcommand{\ma}{\mathcal A}
 \newcommand{\mf}{\mathcal F}
 \newcommand{\mq}{\mathcal Q}
 \newcommand{\mo}{\mathcal O}
 \newcommand{\cm}{\textnormal{CM}}

\title{Stability of Erd\H os-Ko-Rado Theorems in Circle Geometries}
\author{Sam Adriaensen\thanks{Vrije Universiteit Brussel, Pleinlaan 2, 1050 Elsene, Belgium. Email: \url{sam.adriaensen@vub.be}  } \\ {\it Vrije Universiteit Brussel}}
\date{}

\begin{document}

\maketitle

\begin{abstract}
 Circle geometries are incidence structures that capture the geometry of circles on spheres, cones and hyperboloids in 3-dimensional space.
 In a previous paper, the author characterised the largest intersecting families in finite ovoidal circle geometries, except for Möbius planes of odd order.
 In this paper we show that also in these Möbius planes, if the order is greater than 3, the largest intersecting families are the sets of circles through a fixed point.
 We show the same result in the only known family of finite non-ovoidal circle geometries.
 Using the same techniques, we show a stability result on large intersecting families in all ovoidal circle geometries.
 More specifically, we prove that an intersecting family $\mf$ in one of the known finite circle geometries of order $q$, with $|\mf| \geq \frac 1 {\sqrt2} q^2 + 2 \sqrt 2 q + 8$, must consist of circles through a common point, or through a common nucleus in case of a Laguerre plane of even order.
\end{abstract}


\section{Introduction}

In their seminal paper \cite{erdoskorado}, Erd\H os, Ko, and Rado proved the following theorem.

\begin{thm}[{\cite{erdoskorado}}]
Choose integers $k$ and $n$ such that $0 < k \leq n/2$.
Let $\mathcal F$ be a family of subsets of size $k$ of $\set{1,\dots,n}$ such that for each $F, G \in \mathcal F$, $F \cap G \neq \emptyset$.
Then
\[
 | \mathcal F | \leq \binom{n-1}{k-1}.
\]
Moreover, if $k < n/2$, then equality holds if and only if $\mathcal F$ is the family of all $k$-sets through a fixed element of $\set{1,\dots,n}$.
\end{thm}

Since then, there has been a broad interest in the concept of \emph{intersecting families}.
In a general sense, this means the following.
An \emph{incidence structure} is a tuple $(\mp,\mb)$, where we call the elements of $\mp$ \emph{points}, and the elements of $\mb$ \emph{blocks}, and where each block is a subset of $\mp$.
An intersecting family is a set $\mf \subseteq \mb$ such that any two elements of $\mf$ have non-empty intersection.
A typical question in the vein of the Erd\H os-Ko-Rado theorem asks what the size and structure of the largest intersecting families in some incidence structure are.

Extremal combinatorics is concerned with characterising the largest (or smallest) combinatorial objects that satisfy some property.
If all objects of maximum size have the same structure, one could wonder whether there exist large objects which have a significantly different structure.
If such objects do not exist, we speak of a \emph{stability result}.

In \cite{adriaensen2021}, the author proved an Erd\H os-Ko-Rado type theorem in ovoidal circle geometries.
More specifically, it was shown that if the circle geometry is of order $q>2$, and not a Möbius plane of odd order, then the largest intersecting families are exactly the collections of circles through a fixed point, or through a fixed nucleus if the circle geometry is a Laguerre plane of even order.
In this article, we extend this theorem to all known finite circle geometries, and to a stability result.
The two main theorems are as follows.

\begin{thm}
 \label{ThmIntrMain1}
Consider one of the known finite circle geometries of order $q>2$.
In case of a Möbius plane, assume that $q>3$.
The largest intersecting families consist of all circles through a fixed point, or through a fixed nucleus in case of a Laguerre plane of even order.
\end{thm}

\begin{thm}
 \label{ThmIntrMain2}
If $\mf$ is an intersecting family in one of the known finite circle geometries of order $q$, and $|\mf| \geq \frac 1 {\sqrt 2} q^2 + 2 \sqrt 2 q + 8$, then all circles of $\mf$ go through a common point, or a common nucleus in case of a Laguerre plane of even order.
\end{thm}

As a corollary, we obtain stability results of intersecting families in $\pgl(2,q)$ and quadratic polynomials over finite fields.

\bigskip

The structure of the paper is as follows.
In \S \ref{SectionPreliminaries}, we give the necessary background on circle geometries and algebraic graph theory.
In \S \ref{SectionAssoc}, we discuss association schemes related to circle geometries.
In particular, we are interested in their eigenvalues, as they will be essential for the further proofs.
In the next two sections, we prove that if $\mf$ is a large intersecting family in a known circle geometry, then all of the circles of $\mf$ go through a common point $P$.
This is done in two steps.
In \S \ref{SectionMany}, we find a candidate point $P$, that does not lie on ``few'' circles of $\mf$.
In \S \ref{SectionFewOrMany}, we prove that if $\mf$ is a large intersecting family in a circle geometry, then every point lies on either few or many circles of $\mf$.
In \S \ref{SectionMain}, we put these two results together to find a point that lies on many circles of $\mf$, and prove that it in fact lies on all circles of $\mf$.
This yields the two theorems discussed above.
These results also imply stability results of large intersecting families in $\pgl(2,q)$ and in the set of polynomials over a finite field of degree at most two.

\bigskip

The most important step is the idea behind \S \ref{SectionFewOrMany}.
Consider the following situation.
Let $G$ be a regular graph with independence number $\alpha(G)$.
Suppose that we know cocliques $C_1, \dots, C_n$ of size $\alpha(G)$, such that each $C_i$ and its complement form an equitable partition of the graph (which happens for example if $\alpha(G)$ meets Hoffman's ratio bound.)
Take such a coclique $C_i$, and remove all edges from $G$ which have no endpoint in $C_i$.
We are left with a bipartite graph, which must also be biregular.
If the eigenvalues of this bipartite graph can be calculated, we can apply the expander mixing lemma.
The idea is that if we take another large coclique $C'$ in $G$, the bipartite expander mixing lemma implies that if $|C_i \cap C'|$ is large, $C'$ is completely contained in $C_i$.
If we can prove that every large coclique $C'$ has a large intersection with at least one of the $C_i$'s, we have proven that the $C_i$'s are the only large cocliques.

\section{Preliminaries}
 \label{SectionPreliminaries}
 
\subsection{Circle geometries}

In this paper, we will investigate a certain type of incidence structures, called \emph{circle geometries}.
The blocks of a circle geometry are often called \emph{circles}.
To explain what a circle geometry is, we first need to introduce the concept of a \emph{parallel relation}.
This is a partition of the points into so-called \emph{parallel classes}.
We call two points \emph{parallel} if they are in the same parallel class of some parallel relation.

A circle geometry is an incidence structure $(\mp,\mb)$ with at most two parallel relations, such that the following properties hold:
\begin{enumerate}
 \item Given three pairwise non-parallel points, there is a unique circle containing these three points.
 \item Given a circle $c$, a point $P \in c$, and a point $Q \notin c$, not parallel with $P$, there is a unique circle through $P$ and $Q$, which is \emph{tangent} to $c$, i.e.\ it intersects $c$ only in $P$.
 \item Any circle contains a unique point from each parallel class.
 \item Two parallel classes from different parallel relations intersect in a unique point.
 \item Each circle contains at least three points, there exists a circle and a point not on this circle.
\end{enumerate}
A circle geometry with 0, 1, or 2 parallel relations is called respectively a \emph{Möbius plane}, \emph{Laguerre plane}, or a \emph{Minkowski plane}.

\bigskip

In this paper, we only concern ourselves with finite circle geometries, which means that we assume the number of points is finite.
In this case, there exists some integer $q\geq 2$, called the \emph{order} of the circle geometry, such that the following properties hold, where $\rho$ denotes the number of parallel relations.
\begin{itemize}
 \item $|\mp| = \frac{q^2+1-\rho}{q+1-\rho}(q+1)$,
 \item $b := |\mb| = q^3 + (1-\rho)q$,
 \item every circle contains $q+1$ points,
 \item every point lies on $q^2 + (1-\rho)q$ circles,
 \item through any pair of non-parallel points, there are $q+1-\rho$ circles,
 \item every parallel class contains $q+\rho-1$ points.
\end{itemize}

From now on, we denote a circle geometry of order $q$ with $\rho$ parallel classes as a $\cm(\rho,q)$.

\bigskip

Given a point $P$ in a $\cm(\rho,q)$, we define the \emph{residue} at $P$ as the incidence structure $(\mp',\mb')$, where $\mp'$ consists of the points not parallel with $P$, and $\mb'$ consist the circles through $P$ with $P$ removed, and the parallel classes not containing $P$.
If $\rho=2$, every parallel class not containing $P$, contains a unique point parallel with $P$.
These points are also removed from the blocks in $\mb'$.
The residue at any point $P$ is an affine plane of order $q$.

\bigskip

Let $\FF_q$ denote the finite field of size $q$, where $q$ is understood to be a prime power, and let $\pg(n,q)$ denote the projective space arising from the vector space $\FF_q^{n+1}$.
A \emph{quadric} in $\pg(n,q)$ is the set of points whose coordinates satisfy some homogeneous quadratic equation.
To capture the behaviour of a quadric, Buekenhout \cite{buekenhout69} defined a \emph{quadratic set} in $\pg(n,q)$ as a set of points $\mq$ satisfying the following properties.
\begin{enumerate}
 \item Any line that intersects $\mq$ in more than two points, is completely contained in it.
 \item For any point $P \in \mq$, the union of all lines through $P$ which intersect $\mq$ in 1 or $q+1$ points, is a hyperplane or the entire space $\pg(n,q)$.
 We call this the \emph{tangent hyperplane} of $P$.
 \item $\mq$ is not the union of two subspaces.
\end{enumerate}

A point $P$ is called \emph{degenerate} if its tangent hyperplane is the entire space.
A quadratic set is called degenerate if it has a degenerate point.

There are three types of quadratic sets in $\pg(3,q)$, $q>2$.
\begin{enumerate}
 \item An \emph{ovoid} is a set of $q^2+1$ point, no three on a line.
 It was proven by Barlotti \cite{barlotti55}, that the only ovoids in $\pg(3,q)$, $q$ odd, are the so-called elliptic quadrics $\mq^-(3,q)$.
 \item An \emph{oval} is a set of $q+1$ points in $\pg(2,q)$, no three on a line.
 An \emph{oval cone} in $\pg(3,q)$ is a set constructed by taking an oval $\mo$ in a plane, taking some point $R$ not in this plane, and taking the union of all lines through $R$ which intersect $\mo$.
 By Segre's famous result \cite{segre55}, the only ovals in $\pg(2,q)$, $q$ odd, are the quadrics $\mq(2,q)$, also called non-degenerate conics.
 Therefore, there is up to isomorphism only one oval cone in $\pg(3,q)$ if $q$ is odd.
 \item By a result by Buekenhout \cite{buekenhout69}, the only remaining quadratic sets are the hyperbolic quadrics $\mq^+(3,q)$.
 They can be constructed as follows.
 Take three pairwise disjoint lines $l_0,l_1,l_2$ in $\pg(3,q)$.
 There are exactly $q+1$ lines $m_0, \dots, m_q$, which intersect all of the lines $l_0,l_1,l_2$.
 These lines are also pairwise disjoint.
 $l_0,l_1,l_2$ can be extended in a unique way to a set of $q+1$ pairwise disjoint lines $l_0,\dots,l_q$, such that all lines $m_i$ and $l_j$ intersect in a point.
 Then $\bigcup_{i=0}^q m_i = \bigcup_{i=0}^q l_i$ as point sets.
 Such a point set is a hyperbolic quadric.
\end{enumerate}

Given a quadratic set $\mq$ in $\pg(3,q)$, call a plane $\pi$ an \emph{oval plane} if $\pi \cap \mq$ is an oval in $\pi$.
If $\mq$ is non-degenerate, every plane is either a tangent or an oval plane.
If $\mq$ has a degenerate point $P$, the oval planes are exactly the planes missing $P$.
Consider the incidence structure $(\mp,\mb)$, where $\mp$ consist of the non-degenerate points of $\mq$, and $\mb = \sett{\pi \cap \mq}{\pi \textnormal{ an oval plane}}$.
This incidence structure is a circle geometry.
More specifically, it is a Möbius, Laguerre, or Minkowski plane, if $\mq$ is respectively an ovoid, oval cone, or hyperbolic quadric.
A circle geometry that arises in this way from a quadratic set is called \emph{ovoidal}.

Dembowski \cite{dembowski64} and Heise \cite{heise} proved that a Möbius, respectively Minkowski plane of even order must be ovoidal.

\bigskip

Consider an oval cone $\mq$ in $\pg(3,q)$ with vertex $R$ and base $\mo$ in some plane $\pi \not \ni R$.
If $q$ is even, then $\mo$ has a \emph{nucleus} $N$.
This is a point in $\pi$ such that the tangent lines to $\mo$ in $\pi$ are exactly the lines through $N$ (in $\pi$).
Let $\mq_+$ denote $\mq \cup \vspan{R,N}$.
If $L$ denotes the Laguerre plane arising from $\mq$, then let $L_+$ denote the incidence structure $(\mp_+,\mb_+)$ with $\mp_+ = \mq_+ \setminus \set R$, and $\mb_+ = \sett{\pi \cap \mq^+}{\pi \textnormal{ an oval plane}}$.
We call this incidence structure an \emph{extended Laguerre plane}.
Combinatorially, the most important difference between a Laguerre plane and an extended Laguerre plane, is that in an extended Laguerre plane, two distinct circles intersect in either 0 or in 2 points, but never in exactly 1.

Two oval planes intersect in the extended Laguerre plane if and only if they intersect in the original Laguerre plane.
Thus, switching to the extended Laguerre plane makes no (significant) difference for the study of intersecting families.

\bigskip

We now describe the only known non-ovoidal circle geometries.
Let $A$ be a set containing more than three elements, and let $\Pi$ be a sharply 3-transitive set of permutations of $A$.
This means that if $(a_1,a_2,a_3)$ and $(a_4,a_5,a_6)$ are ordered triples of three distinct elements of $A$, there is a unique $\pi \in \Pi$ with $\pi(a_i) = a_{i+3}$ for $i=1,2,3$.
Suppose that $\Pi$ contains the identity element.
For each $\pi \in \Pi$, define its \emph{graph} as $\sett{(a,\pi(a))}{a\in A}$.
Consider the incidence structure $(\mp,\mb)$ with $\mp = A \times A$ and $\mb$ the set of graphs of $\pi \in \Pi$.
This incidence structure is a Minkowski plane, and each finite Minkowski plane can be constructed in this way.

\begin{constr}
 \label{ConstrMinkowski}
Let $q$ be a prime power, and let $\varphi$ be a field automorphism of $\FF_q$.
For each $M \in \pgl(2,q)$, define the map $f_M \in \textnormal{P$\Gamma$L}(2,q)$ as
\[
 f_M(x) = \begin{cases}
  Mx & \text{if } M \in \psl(2,q), \\
  M x^\varphi & \text{otherwise.}
 \end{cases}
\]
Then the set $\Pi_\varphi = \sett{f_M}{M \in \pgl(2,q)}$ is sharply 3-transitive on $\pg(1,q)$.
Denote the corresponding Minkowski plane as $\cm(2,q,\varphi)$.
\end{constr}

If $\varphi = \text{id}$, then the corresponding Minkowski plane is the ovoidal one.
For $q$ even, $\psl(2,q) = \pgl(2,q)$, and the choice of $\varphi$ does not matter (recall that all Minkowski planes of even order are ovoidal).
If $q$ is odd, $|\pgl(2,q) : \psl(2,q)|=2$.
A non-trivial field automorphism $\varphi$ gives rise to a non-ovoidal circle geometry.

\bigskip

For a survey on circle geometries, see e.g.\ Hartmann \cite{hartmann} or Delandtsheer \cite[\S 5]{delandtsheer}.
It is worth mentioning that circle geometries are also called \emph{Benz planes} and Möbius planes are also called \emph{inversive planes}.

\subsection{Graph theory}

All graphs considered in this paper are simple, undirected and without loops.
Let $G=(V,E)$ be a graph.
If two vertices $x$ and $y$ are adjacent, we denote this as $x \sim y$.
The \emph{adjacency matrix} of $G$ is the matrix $A(G)$ whose rows and columns are labelled by the vertices of $G$, such that
\[
 A(G)_{x,y} = \begin{cases}
  1 & \text{if } x \sim y, \\
  0 & \text{otherwise}.
 \end{cases}
\]
We denote the eigenvalues of $A(G)$ as $\lambda_1(G) \geq \lambda_2(G) \geq \dots \geq \lambda_n(G)$.
If $G$ is clear from context, we just write $\lambda_1, \dots, \lambda_n$.
These eigenvalues can convey a lot of information about the graph, as illustrated by the following lemmata.
Let $e(S,T)$ denote the number of edges with an endpoint in $S$ and an endpoint in $T$.

\begin{lm}[{\cite[Theorem 5.1]{haemers}} Bipartite expander mixing lemma]
Let $G$ be a bipartite graph, with bipartition $L$ and $R$.
Take sets $S \subseteq L$ and $T \subseteq R$.
Suppose that every vertex in $L$ has degree $d_L$ and every vertex in $R$ has degree $d_R$.
Then
\[
 \left| e(S,T) - \frac{d_L}{|R|} |S| |T| \right| \leq \lambda_2 \sqrt{|S| |T| \left( 1 - \frac{|S|}{|L|} \right)\left( 1 - \frac{|T|}{|R|} \right)}.
\]
\end{lm}

This lemma has an interesting counterpart in non-bipartite graphs, of which we state a specific case.
For a set $S \subseteq V$, let $e(S)$ denote the number of edges with both endpoints in $S$.

\begin{lm}[{\cite[Theorem 3.5]{haemers}} Non-bipartite expander mixing lemma]
Let $G$ be a $d$-regular graph with $n$ vertices.
Then
\begin{align*}
 \frac d n |S|^2 + \lambda_n |S| \left( 1 - \frac{|S|}n \right)
 \leq  2 e(S) \leq 
 \frac d n |S|^2 + \lambda_2 |S| \left( 1 - \frac{|S|}n \right).
\end{align*}
\end{lm}

\subsection{Association schemes}

\begin{df}
 Let $\ma = \set{A_0, \dots, A_d}$ be a set of non-zero $n \times n$-matrices over $\CC$ with only $0$ and $1$ as entries.
 We call these matrices a \emph{$d$-class association scheme} if the following properties are satisfied.
 \begin{enumerate}
  \item $A_0 = I_n$,
  \item $A_0 + \dots + A_d = J_n$, the $n \times n$ all-one matrix,
  \item the set $\set{A_0,\dots,A_d}$ is closed under transposition,
  \item there exist numbers $p_{ij}^k$, called \emph{intersection numbers}, such that for every $i,j = 0,\dots,d$, $A_i A_j = \sum_{k=0}^d p_{ij}^k A_k$, and $p_{ij}^k = p_{ji}^k$.
 \end{enumerate}
\end{df}

If the rows and columns of the matrices $A_0, \dots, A_d$ of an association scheme are indexed by some set $X$, then there is a natural correspondence between the matrix $A_i$ and the relation $$R_i = \sett{(x,y) \in X \times X}{A_{xy}=1}.$$
We also say that the relations $R_0, \dots, R_d$ are an association scheme, since they clearly convey the same information.

\bigskip

The algebra $\CC[\ma]$ spanned by $A_0, \dots, A_d$ over $\CC$ is called the \emph{Bose-Mesner algebra} of the association scheme.
It readily follows from the definition of an association scheme that this algebra is $(d+1)$-dimensional.

\begin{thm}
The Bose-Mesner algebra of a $d$-class association scheme $\ma$ has a basis $E_0, E_1, \dots, E_d$ of idempotents such that $E_i E_j = \delta_{ij} E_i$, $E_0 = \frac 1 n J_n$, and $\sum_{i=0}^d E_i = I_n$.
\end{thm}

Let $V_i$ denote the column space of $E_i$.
The previous theorem implies that the matrices $A_0, \dots, A_d$ are simultaneously diagonalisable, and each eigenspace of each $A_i$ is a direct sum of some $V_j$'s.

The transition matrix from the $A_i$ basis to the $E_j$ basis is called the \emph{matrix of eigenvalues} (or sometimes the character table) and usually denoted by $P$.

\bigskip

Given a group $G$ of order $n$, let $C_0 = \set{1}, C_1, \dots, C_d$ denote its conjugacy classes.
The relations $R_i = \sett{(x,y) \in G \times G}{y x^{-1} \in C_i}$ constitute an association scheme, sometimes called the \emph{conjugacy class scheme}.
Let $A_0, \dots, A_d$ denote the corresponding $01$-matrices.

The basis of idempotents and their corresponding eigenvalues can be found as follows.
Let $\psi_0 = 1, \psi_1, \dots, \psi_d$ denote the irreducible characters of $G$ over $\CC$.
Define the matrix $E_i$, whose rows and columns are indexed by $G$, as $\displaystyle (E_i)_{x y} = \frac{\psi_i(1)}n {\psi_i(y x^{-1})}$.
Then $E_0, \dots, E_d$ are the basis of idempotents, and $\displaystyle A_i E_j = |C_i| \frac{\overline{\psi_j(c)}}{\psi_j(1)} E_j$ for any $c \in C_i$.
Furthermore, $\dim V_j = \psi_j(1)^2$.

\bigskip

More on association schemes can be found e.g.\ in \cite[\S 3]{godsilmeagher} or \cite[\S 2]{bcn}.
For more on the conjugacy class scheme, see \cite[\S 11]{godsilmeagher}.

\section{Association schemes from ovoidal circle geometries}
 \label{SectionAssoc}

Given a $\cm(\rho,q)$, define the following relations on the circles.
We say that circles $c_1$ and $c_2$ are in relation $R_0, R_1, R_2, R_3$ if $|c_1 \cap c_2|$ equals respectively $q+1$ (i.e.\ $c_1 = c_2$), 1, 2, or 0.
In some circle geometries, these relations constitute an association scheme.
We list these cases, and give the corresponding matrix of eigenvalues, see \cite{adriaensen2021}.

\bigskip

{\bf Möbius planes of even order $\mathbf{q > 2}$.}

\[
P = \begin{pmatrix}
 1 & q^2-1 & \frac{q^2}2 (q+1) & \frac q 2 (q-1)(q-2) \\
 1 & q-1 & -q & 0 \\
 1 & -2 & q \frac{q-1}2 & -(q+1)\frac{q-2}2 \\
 1 & -(q+1) & 0 & q
 \end{pmatrix}
\]

\bigskip

{\bf Ovoidal Laguerre planes of odd order.}

\[
 P = \begin{pmatrix}
1 & q^2-1 & q \frac{q^2-1}2 & q\frac{(q-1)^2}2 \\
1 & -1 & q \frac{q-1}2 & -q \frac{q-1}2 \\
1 & q-1 & -q & 0\\
1 & -(q+1) & 0 & q
 \end{pmatrix}
\]

\bigskip

{\bf Minkowski planes of even order $\mathbf{q > 2}$.}

\[
 P = \begin{pmatrix}
 1 & q^2-1 & q(q+1)\frac{q-2}2 & (q-1)\frac{q^2}2 \\
 1 & q-1 & -q & 0 \\
 1 & -(q+1) & 0 & q \\
 1 & 0 & \frac{q^2-q-2}2 & - (q-1)\frac q 2
 \end{pmatrix}
\]
 
Using the description of an ovoidal Minkowski plane as the graphs of $\pgl(2,q)$, this association scheme and the matrix of eigenvalues can be found in \cite[page 172]{bannai1991}.

\bigskip

{\bf Extended Laguerre plane of even order $\mathbf{q>2}$.}

In this case, we need to drop the relation $R_1$, since it is empty.

\[
 P = \begin{pmatrix}
  1 & \binom{q+2}2 (q-1) & (q-1)^2 \frac q 2 \\
  1 & -\frac{q+2}2 & \frac q 2 \\
  1 & (q+1)\frac{q-2}2 & -(q-1)\frac q 2
 \end{pmatrix}
\]

\bigskip

{\bf Ovoidal Möbius and Minkowski planes of odd order.}

In ovoidal Möbius and Minkowski planes of odd order, the above relations do not constitute an association scheme.
Some of the relations need to be spliced.

Let $q$ be an odd prime power.
Let $\mq^\eps$ denote the elliptic quadric $\mq^-(3,q)$ if $\eps=-1$, and the hyperbolic quadric $\mq^+(3,q)$ if $\eps=+1$.
There exists some symmetric bilinear form $b(x,y)$ on $\FF_q^4$, and an associated quadratic form $\kappa(x) = b(x,x)$, such that $\mq^\eps$ consists of the projective points whose coordinate vectors satisfy $\kappa(x) = 0$.
Define the polarity $\perp$ by mapping a point $P$ with coordinate vector $x$, to the plane $P^\perp$ consisting of the points with coordinate vectors $y$ satisfying $b(x,y)=0$.
Write $P \perp R$ if $P \in R^\perp$ (or equivalently $R \in P^\perp$).
Note that the tangent and oval planes are the planes $P^\perp$ with $P \in \mq^\eps$ respectively $P \notin \mq^\eps$.

\begin{lm}
 \label{LmLPerp}
 Let $l$ be a line in $\pg(3,q)$, not contained in $\mq^\eps$.
 Then
 \[
  | l^\perp \cap \mq^\eps | = \begin{cases}
   2 - | l \cap \mq^\eps | & \text{if } \eps=-1, \\
   | l \cap \mq^\eps | & \text{if } \eps=+1.
  \end{cases}
 \]
\end{lm}

\begin{proof}
Let $i$ denote $|l \cap \mq^\eps|$ and let $x$ denote the number of tangent planes through $l$.
Note that if $Q \in \mq^\eps$, then $Q \in l$ if and only if $l^\perp \subset Q^\perp$, thus $x = |l^\perp \cap \mq^\eps|$.
If $\eps=-1$, then every tangent and oval plane through $l$ contains $1-i$ respectively $q+1-i$ points of $\mq^\eps \setminus l$, and $|\mq^\eps|=q^2+1$.
Hence,
\begin{align*}
 i + x(1-i) + (q+1-x)(q+1-i) = q^2+1 && \Longrightarrow && x = 2-i.
\end{align*}
If $\eps=1$, then every tangent plane through $l$ contains $2q+1-i$ points of $\mq^\eps \setminus l$, and $|\mq^\eps|=(q+1)q(q-1)$.
Hence,
\begin{align*}
 i + x(2q+1-i) + (q+1-x)(q+1-i) = (q+1)q(q-1) && \Longrightarrow && x = i.
\end{align*} 
\end{proof}

Take a point $P \notin \mq^\eps$.
Let $x$ be a coordinate vector for $P$.
Then the other coordinate vectors for $P$ are of the form $\alpha x$, $\alpha \in \FF_q^*$.
For every such $\alpha$, $\kappa(\alpha x) = \alpha^2 \kappa(x)$ is a square if and only if $\kappa(x)$ is a square.
Let $S_q$ and $\overline{S_q}$ denote respectively the non-zero squares and the non-squares of $\FF_q$.
Then we write $\kappa(P)=S_q$ if $\kappa(x) \in S_q$, and $\kappa(P) = \overline{S_q}$ if $\kappa(x) \in \overline{S_q}$.
For a scalar $\alpha$ and subsets $A$ and $B$ of $\FF_q$, let $\alpha A B$ denote the set $\sett{\alpha a b}{a \in A, \, b \in B}$.
Given a conic $C$ in $\pg(2,q)$, we call a point $P \notin C$ \emph{external} or \emph{internal} when $P$ lies on 2 respectively 0 tangent lines to $C$.

\begin{lm}
 Take two points $P$ and $R$ in $\pg(3,q)$, not in $\mq^\eps$, with $P \perp R$.
 Then $P$ is external to $R^\perp \cap \mq^\eps$ if and only if
 \[
  - \kappa(P) \kappa(R) =
  \begin{cases}
   \overline{S_q} & \text{if } \eps = -1, \\
   S_q & \text{if } \eps = +1.
  \end{cases}
 \]
\end{lm}

\begin{proof}
 $P$ is external to $R^\perp \cap \mq^\eps$ if and only if there are two points $Q_1$ and $Q_2$ in $\mq^\eps$ with $Q_1, Q_2 \in R^\perp$ and $P Q_1, P Q_2$ tangent lines.
 Note that since $P \notin \mq^\eps$, $P Q_i$ is a tangent line if and only if $P \in Q_i^\perp$ if and only if $Q_i \in P^\perp$.
 Therefore, $P$ is external to $R^\perp \cap \mq^\eps$ if and only if $(PR)^\perp$ is a 2-secant to $\mq^\eps$.
 
 Let $x$ and $y$ be coordinate vectors for $P$ and $R$ respectively.
 First suppose that $\eps=-1$.
 Then $(PR)^\perp$ is a 2-secant if and only if $PR$ is a 0-secant.
 Equivalently,
 \[
  \kappa(x + \alpha y) = \kappa(x) + 2 b(x,y) \alpha + \kappa(y) \alpha^2 = 0
 \]
 has no solutions.
 Since $P \perp R$, $b(x,y)=0$, and the above is equivalent to $-\kappa(x)/\kappa(y)$ being a non-square.
 Thus, $P$ is external to $R^\perp \cap \mq^\eps$ if and only if $-\kappa(P)/\kappa(R) = - \kappa(P) \kappa(R) = \overline{S_q}$.
 
 The case $\eps=+1$ works analogously, but then $(PR)^\perp$ is a 2-secant if and only if $PR$ is a 2-secant.
\end{proof}

Hence, given two oval planes $P^\perp$ and $R^\perp$ there are two options.
Either $\kappa(P) = \kappa(R)$ and each point of $(PR)^\perp$ is external in $P^\perp$ if and only if it is external in $R^\perp$.
Or $\kappa(P) \neq \kappa(R)$, and each point of $(PR)^\perp$ is external in $P^\perp$ if and only if it is internal in $R^\perp$.
This has us hoping that if we define relations on the circles of our circle geometry depending on the size of the intersection of two circles and whether they correspond to planes $P^\perp$ and $R^\perp$ with $\kappa(P)=\kappa(R)$ or $\kappa(P)\neq\kappa(R)$, these relations constitute an association scheme.
We note that if $P^\perp$ and $R^\perp$ intersect in a unique point of $\mq^\eps$, then all other points of $(PR)^\perp$ are external in $P^\perp$ and $R^\perp$, which implies that $\kappa(P) = \kappa(R)$.
Thus, this would be a 5-class association scheme.

Unfortunately, it is not straightforward to find all intersection numbers of this association scheme.
However, for $\eps=+1$, we can use the alternative representation of the ovoidal Minkowski plane as graphs of $\pgl(2,q)$.
We can then construct the association scheme described above as a so-called subscheme of the conjugacy class scheme of $\pgl(2,q)$, which means that we need to join together different relations of the conjugacy class scheme.

\bigskip

The characters of $\pgl(2,q)$ can e.g.\ be found in $\cite[\S 3]{meagherspiga}$.
We also give the description here for $q$ odd.
The conjugacy classes of $\pgl(2,q)$ are as follows:
\begin{enumerate}
 \item The identity $I_2$.
 \item The matrices with one linearly independent eigenvector.
 All of these matrices are conjugate to $U = \begin{pmatrix} 1 & 1 \\ 0 & 1 \end{pmatrix}$.
 \item The matrices with two linearly independent eigenvectors.
 Each such matrix is conjugate to a matrix of the form $D_x = \begin{pmatrix} 1 & 0 \\ 0 & x \end{pmatrix}$, $x \neq 0, 1$.
 Note that $D_x$ and $D_y$ are conjugate if and only if $x=y$ or $x=y^{-1}$.
 Thus, this gives us $\frac{q-1}2$ conjugacy classes.
 \item The matrices without eigenvectors.
 Each such matrix is conjugate to a matrix of the form $V_r = \begin{pmatrix} 0 & 1 \\ -r^{q+1} & r + r^q \end{pmatrix}$ for some $r \in \FF_{q^2} \setminus \FF_q$.
 The matrices $V_r$ and $V_s$ are conjugate if and only if $r \FF_q^*$ and $s \FF_q^*$ are equal or each others inverses in $\FF_{q^2}^*/\FF_q^*$.
 This gives us $\frac{q+1}2$ conjugacy classes.
\end{enumerate}

Next we describe the irreducible characters of $\pgl(2,q)$.
Define the map 
\[\delta_q: \FF_q^* \to \set{-1,1}: x \mapsto \begin{cases}
 1 & \text{if } x \in S_q, \\
 -1 & \text{otherwise}.
\end{cases}\]
For every group morphism $\gamma: \FF_q^* \to \CC^*$ of order greater than 2, there is an irreducible character $\eta_\gamma$, and $\eta_\gamma = \eta_{\gamma'}$ if and only if $\gamma = \gamma'$ or $\gamma(x) = \gamma'(x^{-1})$.

For every group morphism $\beta:\FF_{q^2}^*/\FF_q^* \to \CC^*$ of order greater than two, there is an irreducible character $\nu_\beta$, and $\nu_\beta = \nu_{\beta'}$ if and only if $\beta=\beta'$ or $\beta(x) = \beta'(x^{-1})$.

The characters can be read in the following table.
For each conjugacy class, we give a representative, which might have a parameter, list the number of parameters that give a representative of a different conjugacy class, and give the size of the conjugacy classes.
Let $i$ denote an element of the unique coset in $\FF_{q^2}^* / \FF_q^*$ of order 2.

\begin{table}[ht]
    \centering
    \begin{tabular}{c|c|c c c c c c}
    \hline
     & Representative & $I_2$ & $U$ & $D_x$, $x \neq -1$ & $D_{-1}$ & $V_r$, $r \FF_q^* \neq i \FF_q^*$ & $V_i$ \\
     & No. & 1 & 1 & $\frac{q-3}2$ & 1 & $\frac{q-1}2$ & 1 \\
     & Size & 1 & $q^2-1$ & $q(q+1)$ & $q \frac{q+1}2$ & $q(q-1)$ & $q \frac{q-1} 2$ \\ \hline 
     Character & No. \\ \hline 
     $\lambda_1$ & 1 & 1 & 1 & 1 & 1 & 1 & 1 \\
     $\lambda_{-1}$ & 1 & 1 & 1 & $\delta_q(x)$ & $\delta_q(-1)$ & $\delta_{q^2}(r)$ & $\delta_{q^2}(i)$ \\
     $\psi_1$ & 1 & $q$ & 0 & 1 & 1 & $-1$ & $-1$ \\
     $\psi_{-1}$ & 1 & $q$ & 0 & $\delta_q(x)$ & $\delta_q(-1)$ & $-\delta_{q^2}(r)$ & $-\delta_{q^2}(i)$ \\
     $\eta_\beta$ & $\frac{q-1}2$ &  $q-1$ & $-1$ & 0 & 0 & $-(\beta(r)+\beta(r^{-1}))$ & $-2 \beta(i)$ \\
     $\nu_\gamma$ & $\frac{q-3}2$ & $q+1$ & 1 & $\gamma(x) + \gamma(x^{-1})$ & $2\gamma(-1)$ & 0 & 0 \\ \hline
    \end{tabular}
    \caption{Character table of $\pgl(2,q)$, $q$ odd. By slight abuse of notation, $\beta(r)$ denotes $\beta(r \FF_q^*)$.}
    \label{TableCharPGL}
\end{table}

We remark that $\delta_q(x)$ and $\delta_{q^2}(r)$ equal 1 if and only if $D_x$ respectively $V_r$ is an element of $\psl(2,q)$.

\bigskip

To obtain the desired subscheme, we need to know which element of $\pgl(2,q)$ corresponds to which oval plane of $\mq^+(3,q)$ in the two isomorphic representations of the ovoidal Minkowski plane.
This correspondence is very straightforward.
We can choose coordinates of $\pg(3,q)$ such that the quadratic form of $\mq^+(3,q)$ is $\kappa(X_1,X_2,X_3,X_4)=X_1 X_4 - X_2 X_3$.
Then the oval planes are the planes of the form $P^\perp$ with $P = (x_1,x_2,x_3,x_4)$ and $x_1 x_4 - x_2 x_3 \neq 0$.
Then $P^\perp$ corresponds to the element $M_P = \begin{pmatrix} x_1 & x_2 \\ x_3 & x_4 \end{pmatrix}$ of $\pgl(2,q)$.
It follows that $\kappa(P) = \kappa(R)$ if and only if $M_R M_P^{-1} \in \psl(2,q)$.
Furthermore, $|P^\perp \cap R^\perp \cap \mq^+(3,q)|$ equals the number of points $Q$ of $\pg(1,q)$ with $M_P Q = M_R Q$, which equals the number of linearly independent eigenvectors of $M_R M_P^{-1}$.

We describe the relations on the circles of the Minkowski plane in the following table.
For each point $P \notin \mq^+(3,q)$ of $\pg(3,q)$, let $c_P = P^\perp \cap \mq^+(3,q)$ denote its corresponding circle.
We describe the condition on $P$ and $R$ for $(c_P,c_R)$ to be in a certain relation.
We also describe to which element of $\pgl(2,q)$ $M_R^{-1} M_P$ must then be conjugate.

\begin{table}[ht]
    \centering
    \begin{tabular}{|c|c|c|}
    \hline
     Relation & Geometric condition & Conjugate element to $M_R M_P^{-1}$ \\ \hline
     $R_0$ & $P=R$ & $I_2$ \\
     $R_1$ & $|c_P \cap c_R| = 1$ & $U$ \\
     $R_2$ & $|c_P \cap c_R| = 2$, $\kappa(P)=\kappa(R)$ & $D_x$, $x \in S_q \setminus \set 1 $ \\
     $R_3$ & $|c_P \cap c_R| = 2$, $\kappa(P) \neq \kappa(R)$ & $D_x$, $x \in \overline{ S_q}$ \\
     $R_4$ & $|c_P \cap c_R| = 0$, $\kappa(P)=\kappa(R)$ & $V_r$, $r \in S_{q^2} \setminus \FF_q $ \\
     $R_5$ & $|c_P \cap c_R| = 0$, $\kappa(P) \neq \kappa(R)$ & $V_r$, $r \in \overline{S_{q^2}}$ \\
    \hline
    \end{tabular}
    \caption{Relations on the ovoidal Minkowski plane of odd order $q$}
    \label{TableRelations}
\end{table}

Let $A_i$ denote the matrix corresponding to relation $R_i$.
The matrices $A_i$ lie in the Bose-Mesner algebra of the conjugacy class scheme of $\pgl(2,q)$.
For each irreducible character $\psi$ of $\pgl(2,q)$, let $E_\psi$ denote its corresponding idempotent.
We will prove that for each $A_i$, its eigenvalue corresponding to $E_{\eta_\beta}$ (resp.\ $E_{\nu_\gamma}$) is independent of $\beta$ (resp.\ $\gamma$).
This implies that the $A_i$ matrices lie in the span of $E_{\lambda_{1}}$, $E_{\lambda_{-1}}$, $E_{\psi_{1}}$, $E_{\psi_{-1}}$, $E_\eta = \sum_\beta E_{\eta_\beta}$, and $E_\nu = \sum_\gamma E_{\nu_\gamma}$.
Since these matrices are also idempotents, this means that the $A_i$ matrices span a 6-dimensional algebra.
It is easy to check that each $A_i$ is symmetric.
Then the fact that they span a 6-dimensional algebra is sufficient to prove that they constitute an association scheme.

Let $D_{S_q}$, $D_{\overline{S_q}}$, $V_{S_{q^2}}$, and $V_{\overline{S_{q^2}}}$
denote the elements of $\pgl(2,q)$ conjugate to respectively $D_x$ with $x \in S_q \setminus \set 1$, $D_x$ with $x \in \overline{S_q}$, $V_r$ with $r \in S_{q^2} \setminus \FF_q$, and $V_r$ with $r \in \overline{S_{q^2}}$.
Then we need to prove that for each $C \in \{D_{S_q},D_{\overline{S_q}},V_{S_{q^2}},V_{\overline{S_{q^2}}}\}$, the value $\sum_{M \in C} \eta_\beta(M)$ does not depend on the choice of $\beta$ and $\sum_{M \in C} \nu_\gamma(M)$ does not depend on $\gamma$.
This only leaves a few non-trivial cases to check.
We use the following lemma, which is a variation on a well-known property.

\begin{lm}
Let $G$ be a finite group, $H \leq G$, and $\varphi:G \to \CC^*$ be a morphism of order greater than $|G:H|$.
Then for each coset $S$ of $H$, $\sum_{x \in S} \varphi(x) = 0$.
\end{lm}

\begin{proof}
There is some element $h \in H$ with $\varphi(h) \neq 1$.
Otherwise, $H \leq \text{Ker}(\varphi)$, and the order of $\text{Im}(\varphi)$ is at most $|G:H|$, contradicting that the order of $\varphi$ is greater than $|G:H|$.
Take a right coset $S = Hg$ of $H$.
Then
\[
 \sum_{x \in S} \varphi(x) = \sum_{x \in H} \varphi(x g) = \sum_{x \in H} \varphi(h x g)
 = \varphi(h) \sum_{x \in H} \varphi(x g) = \varphi(h) \sum_{x \in S} \varphi(x).
\]
Since $\varphi(h)\neq 1$, this implies that $\sum_{x \in S} \varphi(x) = 0$.
Analogous for left cosets of $H$.
\end{proof}

Now take a morphism $\gamma: \FF_q^* \to \CC^*$ or order greater than 2.
\begin{align*}
 \sum_{M \in D_{S_q}} \nu_\gamma(M) & = q(q+1) \sum_{x \in S_q \setminus \set 1} \gamma(x)
 = q(q+1) \left( \sum_{x \in S_q} \gamma(x) - \gamma(1) \right) = - q(q+1), \\
 \sum_{M \in D_{\overline{S_q}}} \nu_\gamma(M) & = q(q+1) \sum_{x \in \overline{S_q}} \gamma(x) = 0.
\end{align*}
Take a morphism $\beta: \FF_{q^2}^*/\FF_q^* \to \CC^*$ of order greater than 2.
\begin{align*}
 \sum_{M \in V_{S_{q^2}}} \eta_\beta(M) & = q(q-1) \sum_{\substack{r \FF_q^* \in \FF_{q^2}^*/\FF_q^* \\ r \in S_{q^2} \setminus \FF_q^*}} \beta(r) = -q(q-1), \\
 \sum_{M \in V_{\overline{S_{q^2}}}} \eta_\beta(M) & = q(q-1) \sum_{\substack{r \FF_q^* \in \FF_{q^2}^*/\FF_q^* \\ r \in \overline{S_{q^2}} }} \beta(r) = 0.
\end{align*}

We can now easily calculate the eigenvalue matrix of this subscheme from the formula for the eigenvalues of the conjugacy class scheme.
This yields
\[
 P = \begin{pmatrix}
  1 & q^2-1 & q \frac{(q-3)(q+1)}4 & q\frac{q^2-1}4 & q \frac{(q-1)^2}4 & q\frac{q^2-1}4 \\
  1 & q^2-1 & q \frac{(q-3)(q+1)}4 & -q\frac{q^2-1}4 & q \frac{(q-1)^2}4 & -q\frac{q^2-1}4 \\
  1 & 0 & \frac{(q-3)(q+1)}4 & \frac{q^2-1}4 & -\frac{(q-1)^2}4 & -\frac{q^2-1}4 \\
  1 & 0 & \frac{(q-3)(q+1)}4 & -\frac{q^2-1}4 & -\frac{(q-1)^2}4 & \frac{q^2-1}4 \\
  1 &  -(q+1) & 0 & 0 & q & 0 \\
  1 & q-1 & -q & 0 & 0 & 0
 \end{pmatrix}
\]

\begin{rmk}
One can also deduce from the character table that the dimensions of the eigenspaces are $1,1,q^2,q^2,\frac{(q-1)^3}2,(q+1)^2\frac{q-3}2$.
If $q=3$, the last eigenspace has dimension 0, since $S_3 \setminus \set 1 = \emptyset$.
On the other hand, relation $R_2$ is empty, since there are no elements of $\pgl(3,q)$ conjugate to $D_x$, $x \in S_3 \setminus \set 1$.
Thus, if $q=3$ we get a 4-class association scheme.
\end{rmk}

\section{A point not lying on few circles}
 \label{SectionMany}

As a first step to characterise large intersecting families $\mf$, we prove that some point lies on ``not few'' circles of $\mf$.

\begin{df}
 \label{DfGi}
 For an incidence structure $(\mp,\mb)$, define the \emph{$i$-intersecting graph} as the graph with vertices $\mb$, where two blocks are adjacent if and only if they intersect in exactly $i$ points.
 We will denote this graph by $G_i$.
\end{df}

The key ingredient in this section is that the $G_1$ graphs for the known circle geometries have good expanding properties.
By this we mean that given two vertices, the number of their common neighbours can not deviate too much from the average number of common neighbours.

In an ovoidal Laguerre plane of even order $q$, the graph $G_1$ is a union of $q$ disjoint copies of the complete graph $K_{q^2}$.
This graph does not have good expanding properties.
This issue can be resolved by switching to the extended Laguerre plane, for which $G_1$ is the empty graph on $q^3$ vertices.
However, this example illustrates that we cannot prove for general circle geometries that $G_1$ has good expanding properties.
Therefore, we will restrict ourselves to the known circle geometries, where we have more control over $G_1$.

\begin{lm}
 \label{LmDegreeG1}
The 1-intersecting graph $G_1$ of a $\cm(\rho,q)$ is $(q^2-1)$-regular.
\end{lm}

\begin{proof}
Take a circle $c$.
For any point $P \in c$, the number of circles that intersect $c$ exactly in $P$ equals the number of lines parallel to (but distinct from) $c \setminus \set P$ in the affine residue at $P$.
This number is $q-1$.
Since there are $q+1$ choices for $P$, there are $(q+1)(q-1)$ circles intersecting $c$ in exactly one point.
\end{proof}

The main idea behind this section is the following elementary counting argument.

\begin{lm}
 \label{LmPointNotOnFew}
Let $\mf$ be an intersecting family in a $\cm(\rho,q)$ of size $f$.
Consider the induced subgraph $G_1[\mf]$ of $G_1$ on $\mf$.
Suppose that $G_1[\mf]$ has at most $E$ edges.
Then there exists a point that lies on at least
\[
 \frac{2 f + (q-1) - \frac{2 E}f}{q+1}
\]
circles of $\mf$.
\end{lm}

\begin{proof}
The sum of the degrees in $G_1[\mf]$ equals at most $2 E$, thus there exists a circle $c$ that has degree at most $\frac{2 E} f$ in $G_1[\mf]$.
Let $n_i$ denote the number of circles of $\mf$ intersecting $c$ in exactly $i$ points.
Then $n_{q+1}=1$, $n_2 = f - 1 - n_1$, and $n_1 \leq \frac{2 E} f$.
By performing a double count, we see that
\[
 |\sett{(P,c') \in c \times \mf}{P \in c'}| = \sum_{i} i \, n_i
 = (q+1) + n_1 + 2(f-1-n_1)
 = 2 f + (q-1) - n_1.
\]
Thus, some point of $c$ lies on at least
\[
 \frac{2 f + (q-1) - n_1}{q+1} \geq \frac{2 f + (q-1) - \frac{2E}f}{q+1}
\]
circles of $\mf$.
\end{proof}

We will use the notation from Lemma \ref{LmPointNotOnFew} throughout this section.

\bigskip

{\bf Extended Laguerre planes.}

In an extended Laguerre plane, $E=0$.

\bigskip

{\bf Ovoidal circle geometries with 3-class association schemes.}

There are three types of ovoidal circle geometries that give rise to 3-class association schemes, namely Möbius and Minkowski planes of even order and Laguerre planes of odd order.
For each of these circle geometries, as can be seen from the previous section, $\lambda_2(G_1) = q-1$.
By the non-bipartite expander mixing lemma,
\begin{align*}
 \label{EqEMLG1}
 2 E \leq \frac{q^2-1}b |\mf|^2 + (q-1) |\mf| \left(1 - \frac{|\mf|}{b} \right)
 = q \frac{q-1}b |\mf|^2 + (q-1) |\mf|,
\end{align*}
with $b = q^3 + (1-\rho)q$.

\bigskip

{\bf Möbius and Minkowski planes of odd order.}

First, we prove that it suffices to check the ovoidal circle geometries.
Recall Construction \ref{ConstrMinkowski}.

\begin{lm}
 \label{LmG1MinkowskiIsomorphic}
Let $q$ be an odd prime power.
The 1-intersecting graphs of $\cm(2,q,\varphi)$ are isomorphic for all field automorphisms $\varphi$.
\end{lm}

\begin{proof}
Take a field automorphism $\varphi$.
We prove that the $G_1$ graph of $\cm(\rho,q,\varphi)$ is isomorphic to the $G_1$ graph of $\cm(\rho,q,\text{id})$.
Take $M \in \pgl(2,q)$.
Let $f_M$ be as in Construction \ref{ConstrMinkowski} with field automorphism $\varphi$.
If $M$ and $N$ are in $\psl(2,q)$, then $f_M$ and $f_N$ are adjacent in $G_1$ if and only if $M N^{-1}$ has a unique fixed point.
If $M$ and $N$ are both not in $\psl(2,q)$, then $f_M$ and $f_N$ are adjacent in $G_1$ if and only if there is a unique projective point $P$ with $M P^\varphi = N P^\varphi$.
This is also equivalent to $M N^{-1}$ having a unique fixed point.
So the map $M \mapsto f_M$ is an embedding of the 1-intersecting graph of $\cm(2,q,\text{id})$ into the 1-intersecting graph of $\cm(2,q,\varphi)$.
Since these graphs have the same valency, it is an isomorphism.
\end{proof}

Now consider the ovoidal Möbius and Minkowski planes of odd order.
It is known, see e.g.\ Hartmann \cite{hartmann}, that the automorphism group of such a circle geometry works transitively on the circles.
Thus, $G_1$ is vertex transitive.
Furthermore, we know that $G_1$ has (at least) two connected components, namely the points $P$ with $\kappa(P)=S_q$ and $\kappa(P)=\overline{S_q}$.
That there are no more than two connected components, can be seen as follows.
In a regular graph, the number of connected components equals the multiplicity of the valency as eigenvalue.
For the ovoidal Minkowski plane of odd order, consider the matrix $P$ of eigenvalues of its 5-class association scheme.
We see that the column of $P$ corresponding to $G_1$ has $q^2-1$ in its first two rows, and both rows correspond to an eigenspace with dimension 1.
For the ovoidal Möbius plane of order $q$, the eigenvalues of the $G_1$ graph restricted to the planes $P^\perp$ with $\kappa(P)=S_q$ can be found in \cite[Table VIII and IX]{bannaihaosong}, in the unique column with valency $q^2-1$ (with $m=2$ in the notation of the paper).

Thus, in both cases, $G_1$ consists of two isomorphic connected components, which means its adjacency matrix (after reordering if necessary) is of the form $A(G_1) = \begin{pmatrix} M & 0 \\ 0 & M \end{pmatrix}$, where $M$ is the adjacency matrix of a connected component.
Moreover, the eigenvalues of $A(G_1)$ are the eigenvalues of $M$ with their multiplicities doubled.
Let $C_1$ denote a connected component of $G_1$.
Then $\lambda_2(C_1) = q-1$ (see the matrix of eigenvalues in the previous section for Minkowski planes or \cite{bannaihaosong} for the Möbius planes).
Now suppose that $|\mf| = f$, and that there are $s$ circles of $\mf$ in $C_1$, then we apply the expander mixing lemma to both components of $G_1$ and find $2E$ is bounded by
\begin{align*}
 &\max_{0 \leq s \leq f} \left[
 \left(\frac {2(q^2-1)}b s^2 + (q-1) s \left(1 - \frac{2s}{b} \right)  \right) + 
 \left(\frac {2(q^2-1)}b (f-s)^2 + (q-1) (f-s) \left(1 - \frac{2(f-s)}{b} \right)  \right) \right]\\
 & = 2 q\frac{q-1} b |\mf|^2 + (q-1) |\mf|.
\end{align*}

This leaves us with the following proposition.

\begin{prop}
 \label{PropMany}
Consider a $\cm(\rho,q)$ that is either ovoidal or from Construction \ref{ConstrMinkowski}.
Let $\mf$ be an intersecting family.
Define
\[
 a = \begin{cases}
  0 & \text{if } \rho=1,\text{ and $q$ is even}, \\
  1 & \text{if } \rho=0,2 \text{ and $q$ is even, or $\rho=1$ and $q$ odd}, \\
  2 & \text{if $\rho=0,2$ and $q$ is odd.}
 \end{cases}
\]
Then some point lies on at least
\[
 \left(2-a \frac{q-1}{q^2+1-\rho}\right) \frac{|\mf|}{q+1}
\]
circles of $\mf$.
\end{prop}

We end with a little note on some of the $G_1$ graphs.

\begin{lm}
 \label{LmDeza}
 Let $q$ be odd.
 Consider a connected component $C_1$ of the 1-intersecting graph of an ovoidal Möbius plane, or a $\cm(2,q,\varphi)$, of odd order $q$.
 Take two distinct circles $c_1$ and $c_2$, which are vertices in $C_1$.
 Then the number of common neighbours of $c_1$ and $c_2$ equals
 \[
  \begin{cases}
   2(q-1) & \text{if } |c_1 \cap c_2| \in \set{1,2}, \\
   2(q+1) & \text{if } |c_1 \cap c_2| = 0.
  \end{cases}
 \]
 Therefore, $C_1$ is a Deza graph\footnote{A \emph{Deza graph} is a graph in which the number of common neighbours of distinct vertices can only take two values, see \cite{erickson+}.}.
\end{lm}

\begin{proof}
First consider an ovoidal Möbius or Minkowski plane.
Let $P^\perp$ and $R^\perp$ be oval planes to $\mq^\eps$ with $\kappa(P)=\kappa(R)$.
Denote $l = (PR)^\perp$.
We count the number of planes intersecting $P^\perp$ and $R^\perp$ in a tangent line to $\mq^\eps$.
Since a tangent plane $Q^\perp$ contains $1+\eps \in \set{0,2}$ $(q+1)$-secants to $\mq^\eps$, $Q$ is the only point in $Q^\perp$ on more than one 1-secant.
Thus, the only tangent plane that we will count is $Q^\perp$ when $l \cap \mq^\eps = \set Q$.

First suppose that $|l \cap \mq^\eps| = 1$.
Then every point of $l\setminus \mq^\eps$ lies on a unique tangent line distinct from $l$ in $P^\perp$ and $R^\perp$.
Also, there are $q-2$ oval planes through $l$ distinct from $P^\perp$ and $R^\perp$.
This yields a total of $2(q-1)$.

Now suppose that $|l \cap \mq^\eps| = i \in \set{0,2}$.
There are $\frac{q+1-i}2$ points on $l$ which are external in $P^\perp$, and therefore also external in $R^\perp$.
Each of these points gives 2 tangent lines in both planes, so 4 planes intersecting both $P^\perp$ and $R^\perp$ in a 1-secant.

Secondly, consider a $\cm(2,q,\varphi)$.
Since, we are working in a connected component of $G_1$, the vertices correspond to graphs of $f_M$ with $M$ the elements of a coset of $\psl(2,q)$.
Thus, the size of the intersection of the circles corresponding to $f_M$ and $f_N$ equals the number of fixed points of $M N^{-1}$.
Hence, it is independent of the choice of $\varphi$.
Now apply the isomorphism to the $G_1$ graph of $\cm(2,q,\text{id})$.
\end{proof}

\section{A point lies on either few or many circles of a large intersecting family}
 \label{SectionFewOrMany}
 
Throughout this section, let $D = (\mp,\mb)$ be a $\cm(\rho,q)$ with $q>2$.
We will prove that if $\mf$ is a large intersecting family in this circle geometry, any point lies on either few or many circles of $\mf$.
We will do this using the following graph.

\begin{df}
 \label{DfGp}
Take a point $P \in \mp$.
Define the sets $L = \sett{c \in \mb}{P \in c}$, and $R = \mb \setminus L$.
We define a graph $G_P$ with bipartition $(L,R)$ where $c_l \in L$ and $c_r \in R$ are adjacent if and only if $c_l \cap c_r = \emptyset$.
\end{df}

\begin{lm}
 \label{LmDelta}
Let $G_P$ be as in Definition \ref{DfGp}, defined on a $\cm(\rho,q)$.
\begin{enumerate}[(1)]
 \item Every vertex in $L$ has degree $\displaystyle \delta = (q+\rho-2) \binom q 2$.
 \item Every vertex in $R$ has degree $\displaystyle \frac{(q+\rho-2)(q+1-\rho)}2$.
\end{enumerate}
\end{lm}

\begin{proof}
(1) Take a circle $c \in L$.
Let $n_i$ denote the number of circles intersecting $c$ in exactly $i$ points.
Since all circles in $L$ contain $P$, $\delta = n_0$.
By Lemma \ref{LmDegreeG1}, $n_1 = q^2-1$.
Through 2 non-parallel points, there are $q+1-\rho$ circles, so $n_2 = \binom{q+1}2 (q-\rho)$.
Obviously, $n_{q+1}=1$ and $\sum_i n_i = b = q^3 + (1-\rho)q$, so
\[
 \delta = n_0 = q^3 + (1-\rho)q - (q^2-1) - \binom{q+1}2(q-\rho) - 1 = (q+\rho-2) \binom q 2.
\]
(2) Take a circle $c \in R$.
Let $n_i$ denote the number of circles through $P$ intersecting $c$ in exactly $i$ points.
Every point $Q$ of $c$ not parallel to $P$ determines a unique tangent circle to $c$ through $P$, so $n_1 = q+1-\rho$.
There are $q-\rho$ other circles through $Q$ and $P$, each intersecting $c$ in 2 points, hence $n_2= \frac 1 2 (q+1-\rho)(q-\rho)$.
The total number of circles through $P$ equals $q(q+1-\rho)$.
Therefore,
\[
 n_0 = q(q+1-\rho) - \frac 1 2 (q+1-\rho)(q-\rho) - (q+1-\rho)
 = \frac{(q+\rho-2)(q+1-\rho)}2.
 \qedhere
\]
\end{proof}

We want to determine the second largest eigenvalue of $G_P$.
The following lemma is well-known, but we include it for completeness sake.

\begin{lm}
 \label{LmLambda2GP}
 Let $L$ and $G_P$ be as in Definition \ref{DfGp}.
 Let $N$ denote the square matrix labelled by $L$, whose $(c_1,c_2)$ entry equals the number of common neighbours of $c_1$ and $c_2$ in $G_P$.
 Then $\lambda_2(G_P) = \sqrt{\lambda_2(N)}$.
\end{lm}

\begin{proof}
Since $G_P$ is bipartite, $A(G_P)$ is of the form $\begin{pmatrix} 0 & M \\ M^t & 0 \end{pmatrix}$ for some matrix $M$ whose rows and columns are labelled by $L$ and $R$ respectively.
Then $A(G_P)^2 = \begin{pmatrix} M M^t & 0 \\ 0 & M^t M\end{pmatrix}$.
The matrices $M M^t$ and $M M^t$ have the same non-zero eigenvalues with the same multiplicity.
Furthermore, the spectrum of $A(G_P)$ is symmetric around 0.
These two properties imply that any non-zero eigenvalue $\lambda$ of $M M^t$ gives non-zero eigenvalues $\pm \sqrt{\lambda}$ of $A(G_P)$ with the same multiplicity.
Therefore, $\lambda_2(G_P) = \sqrt{\lambda_2(M M^t)}$.
To end the proof, one just has to note that $M M^t$ is labelled by the vertices of $L$, and gives the number of common neighbours in $G_P$.
\end{proof}

From now on, let $N$ be as defined in the previous lemma.
We need to compute its spectrum.

\begin{lm}
If $D$ is an extended Laguerre plane of even order, then $\lambda_2(G_P) = \frac 1 2 q \sqrt{q-1}$.
\end{lm}

\begin{proof}
Let $\mq_+$ denote the hyperoval cone used to construct $D$, and denote its vertex by $U$.
Suppose that $c_1$ and $c_2$ are distinct circles through $P$.
Then they intersect is a second point $P_2$.
Let $\pi_1$ and $\pi_2$ be the planes in $\pg(3,q)$ spanned by $c_1$ and $c_2$ respectively.
Denote the intersection line of these planes as $l$.
There are $q-1$ points $Q$ on $l \setminus \set{P,P_2}$.
Through $Q$ there are $\frac q 2$ skew lines $l_1$ to $c_1$ in $\pi_1$.
There are equally many skew lines $l_2$ to $c_2$ in $\pi_2$ through $Q$.
However, if $l_2 = \vspan{l_1,U}\cap\pi_2$, then $\vspan{l_1,l_2}$ is not an (hyper)oval plane.
Thus, there are $(q-1) \frac q 2 \left( \frac q 2 - 1 \right)$ circles disjoint to $c_1$ and $c_2$.

Therefore,
\[
 N = \delta I + \frac{q-2}4 (q-1) q (J - I)
\]
Its eigenspaces are $\vspan \one$ and $\vspan \one ^\perp$.
It follows that the second largest eigenvalue of $N$ equals $\delta - \frac{q-2}4 (q-1) q = \frac{q-1}4 q^2$.
\end{proof}

\begin{lm}
 Let $D$ be a Möbius or Minkowski plane of even order, or an ovoidal Laguerre plane of odd order.
 Then $\lambda_2(G_P) = \frac 1 2 \sqrt{q(q+(2-\rho))(q-(2-\rho))}$.
\end{lm}

\begin{proof}
Consider the 3-class association scheme as described in \S \ref{SectionAssoc} linked to the circle geometries mentioned in the lemma.
Denote the intersection numbers of these association schemes as $p_{ij}^k$.
These numbers can be found in \cite{adriaensen2021}.
The relevant ones are
\begin{align*}
 p^1_{33} = \frac q 4 (q-2+\rho)(q-4+\rho), &&
 p^2_{33} = \frac {q-1} 4 (q-2+\rho)^2.
\end{align*}
To construct $N$, assume that the circles of $L$ are ordered in $q+1-\rho$ blocks of size $q$, each representing a parallel class in the affine residue at $P$.
Then
\[
 N = \delta I + p^1_{33} I_{q+1-\rho} \otimes (J-I)_q + p^2_{33} (J-I)_{q+1-\rho} \otimes J_q.
\]
The eigenspaces of $N$ are spanned by the following vectors.
Here $y_d$ denotes any vector in $\mathbb R^d$, and $x_d$ denotes a vector in $\vspan{\one_d}^\perp$.
\begin{enumerate}
    \item The all-one vector with eigenvalue $\delta + p^1_{33}(q-1) + p^2_{33}(q-\rho)q$,
    \item vectors $y_{q+1-\rho} \otimes x_q$ with eigenvalue $\delta - p^1_{33}$,
    \item vectors $x_{q+1-\rho}\otimes\one_q$ with eigenvalue $\delta + p^1_{33}(q-1) - p^2_{33} q = 0$.
\end{enumerate}
Therefore, $\lambda_2(N) = \delta - p^1_{33}$.
Now apply Lemma \ref{LmLambda2GP}.
\end{proof}

Lastly, we deal with the Möbius and Minkowski planes of odd order.
We can deduce the entries of $N$ from the structure of the $G_1$ graph.
This is a bit convoluted for the ovoidal circle geometries, but allows us to use a unified proof strategy, which also works for the $\cm(2,q,\varphi)$ geometries.

\begin{lm}
 \label{LmBoundLambda}
Let $L$ and $R$ be as in Definition \ref{DfGp}.
Take two distinct circles $c_1$ and $c_2$ in $L$.
Let $n_{11}$ denote the number of circles in $R$ that intersect $c_1$ and $c_2$ in exactly one point, and let $s$ denote $|c_1 \cap c_2|$.
Then the number of common neighbours in $G_P$ of $c_1$ and $c_2$ equals
\[
 \frac{q-s+1}4 (q-3+\rho+s)(q-5+\rho+s) + \frac{n_{11}}4. 
\]
\end{lm}

\begin{proof}
Let $n_{ij}$ denote the number of circles $c \in R$ with $|c_1 \cap c| = i$ and $|c_2 \cap c| = j$, where $i,j \in \set{0,1,2}$.
We want to calculate the number $n_{00}$.
Note that $\sum_{i,j=0}^2 n_{ij} = |R| = q^2(q-1)$, and that $\sum_{i=0}^2 n_{i0} = \sum_{j=0}^2 n_{0j} = \delta$.
Now define the following sets.
\begin{align*}
 N &= \sett{(Q_1,Q_2,c) \in c_1 \times c_2 \times R}{Q_1, Q_2 \in c}, \\
 T_a &= \sett{(Q_1,Q_2,c) \in N}{c_a \cap c = \set{Q_a}} \text{ for } a=1,2.
\end{align*}
Then $|N| = \sum_{i,j} i j n_{ij}$, $|T_1| = n_{11} + 2 n_{12}$, $|T_2| = n_{11} + 2 n_{21}$.
These equations imply that
\[
 4 \sum_{i,j=1}^2 n_{ij} = |N| + |T_1| + |T_2| + n_{11}.
\]
Taking the sum over all $n_{ij}$ where $i$ or $j$ is zero, equals $2 \delta - n_{00}$.
Therefore,
\begin{align*}
 & 2 \delta - n_{00} + \frac 1 4 (|N| + |T_1| + |T_2| + n_{11}) = \sum_{i,j=0}^2 n_{ij} = q^2(q-1) \\
 \Longrightarrow \; & n_{00} = 2(q+\rho-2)\binom q 2 + \frac 1 4 (|N| + |T_1| + |T_2| + n_{11}) - q^2(q-1).
\end{align*}
To finish the proof, we will calculate $|N|$, $|T_1|$, and $|T_2|$ depending on the size of $c_1 \cap c_2$.

\bigskip

\underline{Case 1: $c_1 \cap c_2 = \set P$.}

\begin{enumerate}
    \item $|N| = q(q-\rho)^2$,
    \item $|T_a| = q(q-\rho)$ for $a = 1,2$.
\end{enumerate}
This can be seen as follows: there are $q$ choices for $Q_a \in c_a \setminus \set P$, and $q-\rho$ choices for a point $Q_{3-a} \in c \setminus \set P$, not parallel with $Q_a$.
There are $q-\rho$ circles $c$ through $Q_1$ and $Q_2$, not containing $P$, which yields $|N|$.
There is a unique circle through $Q_{3-a}$ which intersects $c_a$ exactly in $Q_a$, which yields $|T_a|$.

\bigskip

\underline{Case 2: $c_1 \cap c_2 = \set{P,P_2}$.}

\begin{enumerate}
    \item $|N| = (q-1)(q+1-\rho)^2$: We can choose $q-1$ points $Q_1$ on $c_1 \setminus \set{P,P_2}$, $q-1-\rho$ points $Q_2$ on $c_2 \setminus \set{P,P_2}$ not parallel to $Q_2$, and $q-\rho$ circles through $Q_1, Q_2$ and not through $P$.
    If we choose exactly one of the $Q_a$'s equal to $P_2$, then we can choose $Q_{3-a}$ to be any point of $c_{3-a} \setminus \set{P,P_2}$, which also give us $q-\rho$ circles missing $P$.
    We can also choose $Q_1 = Q_2 = P_2$, in which case there are $r$ circles through $Q_1, Q_2$, of which $q+1-\rho$ contain $P$.
    Hence, $|N| = (q-1)(q-1-\rho)(q-\rho) + 2(q-1)(q-\rho) + (q^2 + (1-\rho)q - (q+1-\rho))$.
    \item $|T_a| = (q-1)(q+1-\rho)$ for $a = 1,2$: There are $q-1$ ways to choose $Q_a \in c_a \setminus \set{P,P_2}$, $q-\rho-1$ ways to choose a non-parallel point $Q_{3-a} \in c_{3-a} \setminus \set{P,P_2}$, and one circle through $Q_{3-a}$, which intersects $c_a$ exactly in $Q_a$.
    If we choose $Q_a = P_2$, then we can either choose $Q_{3-a} \neq P_2$, which leaves $q-1$ choices for $Q_{3-a}$, each on a unique circle touching $c_a$ in $Q_a$.
    Or we can choose $Q_1 = Q_2 = P_2$, then there are $q-1$ circles touching $c_a$ in $P_2$ (look in the affine residue at $P_2$), none of which can contain $P$.
    Thus, $|T_a| = (q-1)(q-\rho-1) + (q-1) + (q-1)$. \qedhere
\end{enumerate}
\end{proof}

\begin{df}
Let $q$ be an odd prime power.
First consider an ovoidal Möbius plane of order $q$, with $\kappa$ and $\perp$ as before.
Call a circle \emph{of square type} if its corresponding oval plane is $P^\perp$ with $\kappa(P) = S_q$.
Now consider a $\cm(2,q,\varphi)$.
Call a circle \emph{of square type} if it is the graph of $f_A$ with $A \in \psl(2,q)$.
\end{df}

Note that in the geometric description of an ovoidal Minkowski plane, circles of square type also correspond to the oval plane sections with planes $P^\perp$ for which $\kappa(P) = S_q$.

\begin{lm}
\label{LmEqManyOvalPlanes}
 Half of the circles through any point of an ovoidal Möbius plane, or a $\cm(2,q,\varphi)$, of odd order are of square type.
\end{lm}

\begin{proof}
We first prove this for ovoidal Möbius and Minkowski planes.
Let $\kappa$, $\perp$, and $\mq^\eps$ be as before.
Take a line $l$ and an oval plane $P^\perp$ through $l$.
If $|l \cap \mq^\eps|=1$, then all points of $l \setminus \mq^\eps$ are external to $P^\perp \cap \mq^\eps$, and thus of the same type (i.e.\ square or non-square).
If $|l \cap \mq^\eps|$ equals 0 or 2, then half of the points $R \in l \setminus \mq^\eps$ satisfy $\kappa(R) = S_q$.
Now take a point $Q \in \mq^\eps$, and a line $l$ in $Q^\perp$ not through $Q$.
Then $|l \cap \mq^\eps| = 1+\eps$.
This way we see that $Q^\eps$ contains $\frac{q-\eps}2$ 1-secants through $Q$ on which all points $R \neq Q$ satisfy $\kappa(R)=S_q$, and equally many 1-secants with $\kappa(R) = \overline{S_q}$.
The lemma now follows by applying the polarity.

Next, consider a $\cm(2,q,\varphi)$.
Take a point, which is of the form $(P,Q) \in \pg(1,q)^2$.
Since the lemma holds in the ovoidal Minkowski plane of order $q$, the number of elements $A \in \psl(2,q)$ with $A P = Q$ equals half of the number of circles through $(P,Q)$.
\end{proof}

\begin{lm}
If $D$ is an ovoidal Möbius plane, or a $\cm(2,q,\varphi)$, of odd order $q$, then $\lambda_2(G_P) = \frac12{\sqrt{q(q^2-1)}}$.
\end{lm}

\begin{proof}
Take two distinct circles $c_1$ and $c_2$ through the point $P$.
First we compute the number $n_{11}$ from Lemma \ref{LmBoundLambda}.
If $c_1$ and $c_2$ are of a different type (i.e.\ exactly one of them is of square type), then $n_{11} = 0$.
If $|c_1 \cap c_2| = 1$, then $c_1$ and $c_2$ have $2(q-1)$ common neighbours in $G_1$ by Lemma \ref{LmDeza}.
However, in the affine residue through $P$, we find $q-2$ other circles in the same parallel class as $c_1$ and $c_2$.
Hence, $n_{11} = q$.
If $|c_1 \cap c_2| = 2$ and they are of the same type, $n_{11} = 2(q-1)$, again by Lemma \ref{LmDeza}.
Using Lemma \ref{LmBoundLambda}, we see that the number of common neighbours of $c_1$ and $c_2$ equals
\[
 \begin{cases}
 m_1 = \frac q 4 (q-3+\rho)^2 & \text{if } |c_1 \cap c_2|=1, \\
 m_2 = \frac{q-1}4 \big(q^2 + 2(\rho-2)(q-1)+1 \big) & \text{if $|c_1 \cap c_2| = 2$ and $c_1$ and $c_2$ are of the same type}, \\
 m_3 = \frac{(q-1)^2}4(q-3+2\rho) & \text{otherwise}.
 \end{cases}
\]

Half of the parallel classes in the affine residue through $P$ consist of square type circles.
Arrange the circles through $P$ such that they consist of consecutive blocks of the circles in a parallel class in the affine residue.
Put the circles of square type before the circles of non-square type.
Then
\[
 N = \begin{pmatrix} 1 & 0 \\ 0 & 1 \end{pmatrix} \otimes \left[ I_{\frac{q+1-\rho}2} \otimes \big(\delta I_q + m_1(J-I)_q \big) + (J-I)_{\frac{q+1-\rho}2} \otimes m_2 J_q \right]
 + \begin{pmatrix} 0 & 1 \\ 1 & 0 \end{pmatrix} \otimes m_3 J_{q \frac{q+1-\rho}2}.
\]
The eigenspaces are spanned by the following eigenvectors.
Here $y_d$ denotes a vector from $\RR^d$, and $x_d$ denotes a vector from $\vspan{\one_d}^\perp$.
\begin{enumerate}
 \item The all-one vector with eigenvalue $\delta \frac{(q+\rho-2)(q+1-\rho)}2$,
 \item vectors $y_{q+1-\rho} \otimes x_q$, with eigenvalue $\delta - m_1$,
 \item vectors $y_2 \otimes x_{\frac{q+1-\rho}2} \otimes \one_q$ with eigenvalue $\delta + (q-1) m_1 - q m_2 = 0$,
 \item vectors $x_2 \otimes \one_{q \frac{q+1-\rho}2}$ with eigenvalue $\delta + (q-1)m_1 + q \frac{q-1-\rho}2 m_2 - q \frac{q+1-\rho}2 m_3$.
\end{enumerate}
From this, the spectrum of $N$ can be calculated.
Then apply Lemma \ref{LmLambda2GP}.
\end{proof}

\begin{prop}
 \label{PropFewOrMany}
 Let $\mf$ be an intersecting family in an ovoidal $\cm(\rho,q)$ or a $\cm(2,q,\varphi)$, and let $P$ be a point.
 In case of a Laguerre plane of even order, switch to the extended Laguerre plane.
 Define $S = \sett{c \in \mf}{P \in c}$ and $T = \mf \setminus S$, and define
 \[
  \lambda = \begin{cases}
   q(q^2-1) & \text{if $q$ is odd,} \\
   q(q-2+\rho)(q+(1-\rho)(2-\rho)) & \text{if $q$ is even.} 
  \end{cases}
 \]
 Then
 \[
  |S| \, |T| \leq \left( \frac q {q+\rho-2} \right)^2 \lambda \left(1 - \frac{|S|}{q(q+1-\rho)} \right) \left(1 - \frac{|T|}{q^2(q-1)} \right).
 \]
\end{prop}

\begin{proof}
Consider the graph $G_P$ from Definition \ref{DfGp}.
Since $\mf$ is an intersecting family, $e(S,T) = 0$.
Now apply the bipartite expander mixing lemma to $G_P$ on the sets $S$ and $T$.
The eigenvalue $\lambda_2(G_P)$ can be found in the above lemmata.
\end{proof}

\begin{rmk}
 \label{RmkWorstConstant}
The constant $\left( \frac q {q+\rho-2} \right)^2 \lambda$ above is at most $\frac{q^3(q^2-1)}{(q-2)^2}$, which occurs for Möbius planes of odd order.
\end{rmk}

\section{The main theorems}
 \label{SectionMain}

First we prove Erd\H os-Ko-Rado type results for the known circle geometries, with very mild conditions on the order.
Then we prove a stability result in these circle geometries.
We will use the following lemma.

\begin{lm}
 \label{LmHM}
 Suppose that $\mf$ is an intersecting family in a $\cm(\rho,q)$.
 If a point $P$ lies on more than $\binom{q+2-\rho}2$ circles of $\mf$, then $P$ lies on all circles of $\mf$.
\end{lm}

\begin{proof}
 Suppose that $\mf$ contains a circle $c$ not through $P$.
 Then every circle through $P$ in $\mf$ must intersect $c$.
 Every two points of $c$ not parallel to $P$ determine a unique circle through $P$, and every point $Q$ of $c$ not parallel to $P$ determines a unique circle through $P$ tangent to $c$ in $Q$.
 Therefore, the number of circles through $P$ intersecting $c$ equals $\binom{q+1-\rho}2 + q+1-\rho = \binom{q+2-\rho}2$.
\end{proof}

\subsection{Characterisation of the largest intersecting families}

\begin{thm}
Let $\mf$ be an intersecting family in an ovoidal Möbius plane of odd order $q>3$.
Then $|\mf| \leq q(q+1)$, with equality if and only if $\mf$ consists of all circles through a fixed point.
\end{thm}

\begin{proof}
Let $\mf$ be an intersecting family of size $q(q+1)$.
By Proposition \ref{PropMany}, there exists a point $P$ that lies on $s \geq 2q-1$ circles of $\mf$.
Apply Proposition \ref{PropFewOrMany} to this point $P$.
This yields
\begin{align*}
 s(q(q+1)-s) & \leq \frac {q^2} {(q-2)^2} q(q^2-1) \left(1-\frac s{q(q+1)} \right) \left( 1-\frac{q(q+1)-s}{q^2(q-1)} \right) \\
 & = \frac 1 {(q-2)^2} \big(q(q+1)-s \big) \big(q^2(q-1) - q(q+1) + s \big)
\end{align*}
Suppose that $s < q(q+1)$.
Then this inequality can be rewritten as
\[
 s \big( (q-2)^2 - 1 \big) \leq q^3-2q^2-q.
\]
Since $q \geq 5$, this implies that $s < q+4$, contradicting $s \geq 2q-1$.
Therefore, $s$ must equal $q(q+1)$, or in other words $\mf$ must consist of all circles through $P$.

It now follows that if $\mf$ is an intersecting family of size at least $q(q+1)$, $\mf$ contains a subfamily of size $q(q+1)$ which must consist of all circles through a fixed point.
By Lemma \ref{LmHM}, $\mf$ must equal this subfamily.
\end{proof}

\begin{rmk}
 \label{RmkMöbiusQ=3}
There is a unique Möbius plane of order 3, see e.g.\ \cite{thas1994}.
This Möbius plane is ovoidal, and constructed from plane sections with $\mq^-(3,3)$.
Take two oval planes $P^\perp$ and $R^\perp$.
Suppose that $P$ and $R$ have coordinate vectors $x$ and $y$ respectively.
Then $(PR)^\perp$ is a $0$-secant to $\mq^-$ if and only if $PR$ is a 2-secant, by Lemma \ref{LmLPerp}.
This means that the points with coefficients $x \pm y$ must both be in $\mq^-$ or equivalently that for $\alpha = \pm 1$
\[
 \kappa(x + \alpha y) = \kappa(x) + 2 \alpha b(x,y) + \alpha^2 \kappa(y)
 = \kappa(x) + 2 \alpha b(x,y) + \kappa(y) = 0.
\]
This happens if and only if $\kappa(x) + \kappa(y) = 0$ and $b(x,y)=0$.
Therefore, the sets $\sett{P^\perp}{\kappa(P)=S_3}$ and $\sett{P^\perp}{\kappa(P)=\overline{S_3}}$ are intersecting families of size $\frac{q(q^2+1)}2 = 15$.
Note that $q(q+1)=12$.
Furthermore, it is well-known that the only cocliques in a connected regular bipartite graph of maximal size are the bipartition classes.
Thus, these two families are the only cocliques of size 15 in the $G_0$ graph, proving they are the only two intersecting families of size 15 in the $\cm(0,3)$.
\end{rmk}

\begin{thm}
Let $\mf$ be an intersecting family in a $\cm(2,q,\varphi)$ from Construction \ref{ConstrMinkowski}.
Then $|\mf| \leq q(q-1)$, with equality if and only if $\mf$ consists of all circles through a fixed point.
\end{thm}

\begin{proof}
If $\varphi=\text{id}$, the corresponding Minkowski plane is ovoidal, and the theorem has been proven by Meagher and Spiga \cite{meagherspiga}.
So suppose that $\varphi \neq \text{id}$.
Then $q$ cannot be prime, hence $q \geq 9$.
Consider an intersecting family $\mf$ of size $q(q-1)$.
By Proposition \ref{PropMany}, there exists a point $P$ that lies on $s \geq 2q-5$ circles of $\mf$.
Then by Proposition \ref{PropFewOrMany},
\[
 s (q(q-1) - s) \leq q(q^2-1) \left(1- \frac s {q(q-1)} \right) \left( 1 - \frac{q(q-1) - s}{q^2(q-1)} \right)
 < (q+1)( q(q-1) - s).
\]
Thus, if $s < q(q-1)$, then this inequality implies that $s < q+1$, contradicting $s\geq 2q-5$ since $q\geq9$.
Therefore, the only intersecting families of size $q(q-1)$ are the ones containing all circles through a fixed point, and the theorem follows.
\end{proof}

Together with the main theorems of \cite{adriaensen2021}, this proves Theorem \ref{ThmIntrMain1}.

\subsection{Stability result}

\begin{thm}
Consider an intersecting family $\mf$ in one of the known finite circle geometries, i.e.\ an ovoidal circle geometry or Construction \ref{ConstrMinkowski}, of order $q$.
In case of an ovoidal Laguerre plane of even order, switch to the extended Laguerre plane.
If $|\mf| \geq \frac 1 {\sqrt 2} q^2 + 2 \sqrt 2 q + 8$, then $\mf$ consists of circles through a common point.
\end{thm}

\begin{proof}
Denote the size of $\mf$ by $f$.
We may assume that $q \geq 11$, otherwise $f$ exceeds the size of the largest intersecting family.
This implies that the largest intersecting families are the sets of all circles through a fixed point.
Then by Proposition \ref{PropMany}, some point $P$ lies on $s \geq 2 \frac q{(q+1)^2} f$ circles of $\mf$.
By Proposition \ref{PropFewOrMany} and Remark \ref{RmkWorstConstant}, we obtain the inequality
\[
 s (f-s) \leq q^3 \frac{q^2-1}{(q-2)^2}.
\]
The inequality does not hold for $s = 2 \frac q{(q+1)^2} f$ and for $s = \binom{q+2}2$ given the assumptions on $f$ and $q$.
Since it is quadratic in $s$, this implies that it doesn't hold for any intermediary values.
Hence, $P$ lies on more than $\binom{q+2}2$ circles of $\mf$, and by Lemma \ref{LmHM} on all circles of $\mf$.
\end{proof}

\subsection{Corollaries in other incidence structures}

An intersecting family in $\pgl(2,q)$ is a set $\mf$ of elements of $\pgl(2,q)$ such that for each $M,N \in \mf$ there exists a point $P$ of $\pg(1,q)$ with $M P = N P$.
We know that this is equivalent to an intersecting family in the ovoidal Minkowski plane of order $q$.

\begin{crl}
An intersecting family in $\pgl(2,q)$ of size at least $\frac 1 {\sqrt 2}q^2 + 2 \sqrt 2 q + 8$ lies in a coset of the stabiliser of some point of $\pg(1,q)$.
\end{crl}

Let $U_{2,q}$ denote the set of all polynomials of degree at most 2 over $\FF_q$.
For $f(X) = a X^2 + b X + c$, we define $f(\infty) = a$.
For each $f \in U_{2,q}$, define its graph $\sett{(x,f(x))}{f \in \FF_q \cup \set \infty}$.
Consider the incidence structure $(\mp,\mb)$ with $\mp=(\FF_q \cup \set \infty) \times \FF_q$ and $\mb$ the set of graphs of $f \in U_{2,q}$.
This incidence structure is an ovoidal Laguerre plane.
If $q$ is even, we can go to the extended Laguerre plane by adding the points $\set{-\infty}\times\FF_q$, and adding $(-\infty,b)$ to the graph of $aX^2+bX+c$.

We call $\mf \subseteq U_{q,2}$ an intersecting family if for every $f,g \in \mf$ there exists an $x \in \FF_q \cup \set{\pm \infty}$ with $f(x)=g(x)$.

\begin{crl}
An intersecting family in $U_{2,q}$ of size at least $\frac 1 {\sqrt 2} q^2 + 2 \sqrt 2 q + 8$ consists of functions $f$ with $f(x)=y$ for some fixed $x$ and $y$ in $\FF_q$.
\end{crl}

\begin{proof}
Analogous to the first part of the proof of \cite[Theorem 6.2]{adriaensen2021}.
\end{proof}

\section{Concluding remarks}

In this paper, we characterised large intersecting families in the known finite $\cm(\rho,q)$ as sets of circles through a common point or nucleus.
One could wonder whether the bound can't be pushed further.
If an Erd\H os-Ko-Rado result holds in an incidence structure $(\mp,\mb)$, such as a $\cm(\rho,q)$, then often the largest intersecting families $\mf$ with $\bigcap \mf = \emptyset$ are of the following form, where $P$ and $B$ are a non-incident point and block respectively.
\[
 \sett{B' \in \mb}{P \in B', \, B' \cap B \neq \emptyset} \cup \set B
\]
If this is true in some incidence structure, it is often referred to as a Hilton-Milner type result, after the paper by Hilton and Milner \cite{hiltonmilner}, which proves such a result where $\mb$ is the set of all subsets of $\mp$ of size $k$.
In a circle geometry, such a Hilton-Milner type family would be of size $\binom {q+2-\rho} 2 + 1$, as can be seen from Lemma \ref{LmHM}.
To prove a Hilton-Milner result, we would need to improve the constant $\frac 1 {\sqrt 2}$ in the current bound to $\frac 1 2$.

A set of blocks $\mf$ is called a \emph{t-intersecting family} if any two elements of $\mf$ intersect in at least $t$ elements.
In general, the bounds on the size of a 2-intersecting family in a $\cm(\rho,q)$ are of the form $\frac 1 2 q^2 + \mo(q)$, but no 2-intersecting family with size larger than $2q$ is known.
To the author's best knowledge, ovoidal Laguerre planes of even order are the only type of circle geometry for which better bounds are known.
There, the largest 2-intersecting families have size $q$.
Thus, proving a Hilton-Milner type result for circle geometries could have interesting consequences for the size of the largest 2-intersecting families.

Recently, Maleki and Razafimahatratra \cite{malekirazafimahatratra} proved that the only intersecting families in GL$(2,q)$ acting on $\FF_q^2 \setminus \set \zero$ are stabilisers of a point or a hyperplane.
This is in contrast to the existence of Hilton-Milner type families in $\pgl(2,q)$ acting on $\pg(1,q)$, or equivalently in the ovoidal Minkowski planes.

\bigskip

One could also wonder whether the main theorem of this paper can be proven purely from the combinatorial definition of circle geometries, instead of having to use the structure of the known examples.
As was noted in \S \ref{SectionMany}, a key ingredient in the proof is that the 1-intersecting graphs of the known circle geometries have good expanding properties, except for ovoidal Laguerre planes of even order.
The fact that this doesn't hold for these Laguerre planes, might suggest that some extra structure needs to be imposed on the circle geometries.

\bigskip

It would also be nice to show that the relations as defined in the second column of Table \ref{TableRelations} yield a 5-class association scheme for ovoidal Minkowski planes of odd order (or a 4-class scheme if $q=3$).
It has been checked that this is the case for $q \leq 13$ by computer.

\bigskip

Lastly, we remark that it might be interesting to explore further applications of the technique of this paper to prove (stability of) Erd\H os-Ko-Rado results in other incidence structures.
The scope of our technique to find a point on ``not few'' circles of the intersecting family is probably limited to incidence structures in which the size of the intersection of two blocks can only take a very limited number of values.

\bigskip

{\bf Acknowledgements.}
The author would like to thank his supervisor Jan De Beule, and his colleague Sam Mattheus for lending his expertise in algebraic combinatorics.

\bibliographystyle{alpha}
\bibliography{ref.bib}

\end{document}